\documentclass[12pt]{amsart}

\usepackage{amsmath}
\usepackage{amssymb}
\usepackage{amsthm}
\usepackage{amscd}
\usepackage{xypic}
\usepackage{delarray}

\headsep=1cm \numberwithin{equation}{section}
\newtheorem{theorem}{Theorem}[section]
\newtheorem{lemma}[theorem]{Lemma}
\newtheorem{proposition}[theorem]{Proposition}

\theoremstyle{definition}
\newtheorem{definition}[theorem]{Definition}
\newtheorem{remark}[theorem]{Remark}

\newtheorem{convention and reminder}[theorem]{Convention and Reminder}
\newtheorem{convention and remark}[theorem]{Convention and Remark}
\newtheorem{definition and remark}[theorem]{Definition and Remark}

\newtheorem{reminders and definition}[theorem]{Reminders and Definition}

\newtheorem{notation and remarks}[theorem]{Notation and Remarks}
\newtheorem{notation and remark}[theorem]{Notation and Remark}

\def\P{{\mathbb P}}

\newcommand\Ker{\operatorname{\Ker}}

\title[Completely decomposable defining equations of points]{On completely decomposable defining equations of points in general position in $\P^n$}

\begin{document}

\author{Jaeheun Jung}
\address{Department of Mathematics, Korea University, Seoul 136-701, Korea}
\email{wodsos@korea.ac.kr}

\author{Euisung Park}
\address{Department of Mathematics, Korea University, Seoul 136-701, Korea}
\email{euisungpark@korea.ac.kr}

\date{Seoul, \today}



\maketitle \thispagestyle{empty}

\begin{abstract}
The study of the defining equations of a finite set in linearly general position has been actively attracted since it plays a significant role in understanding the defining equations of arithmetically Cohen-Macaulay varieties. In \cite{T}, R. Treger proved that $I(\Gamma)$ is generated by forms of degree $\leq \lceil \frac{|\Gamma|}{n}\rceil$. Since then, Treger’s result have been extended and improved in several papers.

The aim of this paper is to reprove and improve the above Treger’s result from a new perspective. Our main result in this paper shows that $I(\Gamma)$ is generated by the union of $I(\Gamma)_{\leq \lceil \frac{|\Gamma|}{n}\rceil -1}$ and the set of all completely decomposable forms of degree $\lceil \frac{|\Gamma|}{n}\rceil$ in $I(\Gamma)$. In particular, it holds that if $d \leq 2n$ then $I(\Gamma)$ is generated by quadratic equations of rank $2$. This reproves Saint-Donat’s results in \cite{SD1} and \cite{SD2}.
\end{abstract}

\section{Introduction}
\noindent Let $X \subset \P^r$ be a nondegenerate irreducible projective
variety defined over an algebraically closed field $k$ of arbitrary
characteristic. Let $d$ and $n$ denote respectively the degree and
the codimension of $X$ in $\P^r$. It is well known that if
$\rm{char}(k)=0$ and $\Lambda =\P^n$ is a general $n$-dimensional
subspace of $\P^r$, then the scheme-theoretic intersection $X \cap
    \Lambda \subset \P^n$ is a reduced finite set of $d$ points \textit{in
    linearly general position}; that is, any $n+1$ points of $X \cap
    \Lambda$ spans the whole space $\Lambda$. For this reason, finite sets of
points in linearly general position have played a significant role
in finding an upper bound for the degrees of the defining equations
of $X \subset \P^r$ when it is arithmetically Cohen-Macaulay.
Keeping this approach in mind, let
\begin{equation*}
\Gamma \subset \P^n
\end{equation*}
be a finite set of $d$ points in linearly general position. Also let $S$
and $I(\Gamma)$ be respectively the homogeneous coordinate ring $k[x_0 ,x_1 , \ldots , x_n ]$ of $\P^n$ and the homogeneous ideal of $\Gamma$ in $S$. Then one can easily show that $I(\Gamma)$ is generated by forms of degree $\leq m+1$. Indeed, this
comes from the fact that $\Gamma$ is $m$-normal and hence $(m+1)$-regular in the sense of Castelnuovo-Mumford (cf. \cite[Proposition 1.1]{GM}). Maybe the first result improving this elementary fact is due to B. Saint-Donat, who proved in \cite{SD1} that $I(\Gamma)$ is generated by quadratic equations if $d \leq 2n$. Later, by using an old argument essentially due to K. Petri in \cite{Pe} and B. Saint-Donat, R. Treger in \cite{T} has shown that if $d \leq mn$, then $I(\Gamma)$ is generated by forms of degree $\leq m$. Since then, Treger's result have been extended and improved in several papers (cf. \cite{MV}, \cite{TV}, etc).

The aim of this paper is to reprove and improve the above Treger's
result from a new perspective. To state our result precisely, we
require some notation and definitions.

\begin{definition}
$(1)$ A homogeneous polynomial in $S = k[x_0 , x_1 , \ldots , x_n ]$ is said to be \textit{completely decomposable} if it can be written as a product of linear forms.

\noindent $(2)$ Let $\Gamma \subset \P^n$ be a finite set and $\ell \geq 1$ an integer. Then we denote by $\Phi (\Gamma)_{\ell}$ the set of all completely decomposable polynomials in $I(\Gamma)_{\ell}$. That is,
\begin{equation*}
\Phi (\Gamma)_{\ell} = I(\Gamma)_{\ell} \cap \{ h_1 \cdots h_{\ell} ~|~ h_1 , \ldots , h_{\ell} \in S_1 - \{0\} \}.
\end{equation*}
\end{definition}
Now, let $\Gamma \subset \P^n$ be a finite set of $d$ points in linearly general position. When $d \leq mn$, one can easily create the elements in $\Phi (\Gamma)_m$.
Namely, let
\begin{equation*}
    \Gamma = \Gamma_1 \cup \cdots \cup \Gamma_m
\end{equation*}
be a partition of $\Gamma$ such that $|\Gamma_i | \leq n$ for all $1 \leq i \leq m$. Then
\begin{equation*}
    \{ h_1 \cdots h_m ~|~ h_i \in I(\Gamma_i )_1 - \{0\} \quad \mbox{for
        each} \quad 1 \leq i \leq m \}
\end{equation*}
is a nonempty subset of $\Phi (\Gamma)_m$. Any element in
$\Phi (\Gamma)_m$ can be constructed by this method. Note that $\Gamma$ is cut out by $\Phi (\Gamma)_m$ set-theoretically. For example, see the proof of \cite[Theorem 1.4]{H} where it is shown that if $d \leq 2n$ then the common zero set of $\Phi (\Gamma)_2$ is exactly equal to $\Gamma$. Along this line, one can naturally ask whether $\Gamma$ is ideal-theoretically cut out by $\Phi (\Gamma)_m$ and $I(\Gamma)_{\leq m-1}$. In \cite{SD1} and \cite{SD2}, B. Saint-Donat proved that if $d \leq 2n$ then $I(\Gamma)$ is generated by $\Phi (\Gamma)_2$, or equivalently, $I(\Gamma)$ is generated by quadratic equations of rank $2$. By applying this result to curves, he also proved that if $\mathcal{C} \subset \P^{n+1}$ is a linearly normal projective integral curve of arithmetic genus $g$ and degree $d \geq 2g+2$ then $I(\mathcal{C})$ is generated by quadratic equations of rank $3$ and $4$.

Along this line, our main result in the present paper is the following theorem.

\begin{theorem}\label{thm:main}
Let $\Gamma \subset \P^n$ be a finite set of $d$ points in linearly general position. If $d \leq mn$, then
\begin{equation*}
I(\Gamma) = \langle I(\Gamma)_{\leq m-1} ,\Phi (\Gamma)_m \rangle.
\end{equation*}
In particular, $I(\Gamma)$ is generated by forms of degree $\leq m$.
\end{theorem}

Obviously, this result reproves Saint-Donat's result and improves
Treger's result. That is, if $m=2$ then Theorem \ref{thm:main} says that $I(\Gamma)$ is generated by $\Phi (\Gamma)_2 $. Thus Theorem \ref{thm:main} reproves Saint-Donat's result in \cite{SD1} and \cite{SD2}. Also for arbitrary $m \geq 2$, Theorem \ref{thm:main} says that $I(\Gamma)$ is generated by $I(\Gamma)_{\leq m}$ and some completely decomposable $m$-forms spans $I(\Gamma)_m$ as a $k$-vector space. Hence it improves
Treger's result in \cite{T}.

The proof of Theorem \ref{thm:main} uses a completely different idea from Saint-Donat's and Treger's approaches. More precisely, we first prove that $I(\Gamma)_m$ is generated by $\Phi (\Gamma)_m$ as a $k$-vector space by using some geometric properties of \textit{the variety of completely decomposable $m$-forms}
\begin{equation*}
\mbox{Split}_m (\P^n) \subset \P (S_m ) = \P^{{n+m} \choose {n} -1}
\end{equation*}
(cf. \cite{AB} and \cite{A}). Indeed, regarding $\P (I(\Gamma)_m )$ as a subspace of  $\P (S_m )$, the desired statement that $\Phi (\Gamma)_m$ generates $I(\Gamma)_m$ as a $k$-vector space is equivalent to that the intersection
\begin{equation*}
[\Phi (\Gamma)_m ] := \mbox{Split}_m (\P^n) \cap \P (I(\Gamma)_m ) = \{ [F] ~|~ F \in \Phi (\Gamma)_m \}
\end{equation*}
is a nondegenerate subset of $\P (I(\Gamma)_m )$. Briefly speaking, the second statement can be shown by using a few projective geometric properties of the split variety $\mbox{Split}_m (\P^n)$. For details, see $\S~2$ and $\S~3$. Then, in $\S~4$ and $\S~5$, we prove that $I(\Gamma)_{m+1}$ is equal to the $k$-vector space generated by the set
\begin{equation*}
S_1 \Phi (\Gamma)_{m} := \{ h F ~|~ h \in S_1 \quad \mbox{and} \quad F \in \Phi (\Gamma)_{m} \}.
\end{equation*}
This completes the proof of Theorem \ref{thm:main} since we know already that $\Gamma$ is $(m+1)$-regular in the sense of Castelnuovo-Mumford.\\

\noindent {\bf Acknowledgement.} The second named author was
supported by the Korea Research Foundation Grant funded by the
Korean Government (NRF-2016R1D1A1A09918462).

\section{Completely decomposable equations I : $d=mn$ case}
\noindent Let $\Gamma \subset \P^n$ be a finite set of $d$ points in linearly general position. Through this section and the next section, we will show that if $d \leq \ell n$ then $I(\Gamma)_{\ell}$ is, as a $k$-vector space, generated by its subset $\Phi (\Gamma)_{\ell}$ consisting of all completely decomposable $\ell$-forms.

This section is devoted to giving a proof of the following theorem.

\begin{theorem}\label{thm:d=mn}
Let $\Gamma \subset \P^n$ be a finite set of $d=mn$ points in
linearly general position for some $m \geq 2$. Then $\Phi
(\Gamma)_m$ generates $I(\Gamma)_m$ as a $k$-vector space.
\end{theorem}

To this aim, we begin with the following well-known fact. We will give a brief proof due to the lack of references.

\begin{proposition}\label{prop:lCM flatness}
Let $X\subseteq \P^r$ be a nondegenerate irreducible projective variety of degree $d$ and codimension $c \geq 1$. Suppose that $X$ is locally Cohen-Macaulay. If $\Lambda= \P^c$ is a subspace of $\P^r$ such that $X \cap \Lambda$ is a finite set of $d$ points, then $X \cap \Lambda$ spans $\Lambda$.
\end{proposition}

\begin{proof}
Let $N = \P^{c-t}$ be the span in $\Lambda = \P^c$ of the finite set $X \cap \Lambda$. Suppose that $t \geq 1$ and choose a subspace $M=\P^{c-t-1}$ of $N$ which does not meet with $X \cap \Lambda$. Consider the linear projection map
\begin{equation*}
{\pi}_M : \P^r - M \longrightarrow \P^{r-c+t}
\end{equation*}
from $M$. This defines a finite morphism
\begin{equation*}
f : X \longrightarrow \P^{r-c+t}.
\end{equation*}
Let $Y := f(X)$ be the image of $X$ in $\P^{r-c+t}$. Then it holds that
\begin{equation*}
d = \deg(f)\times \deg(Y)
\end{equation*}
and hence $\deg (f) \leq d$. On the other hand, note that $f(X \cap \Lambda )$ is a point. This implies that $\deg (f)$ is at least $d$ since $X$ is locally Cohen-Macaulay (cf. \cite[Exercise 18.17]{E}). Thus it is shown that $\deg (f)=d$ and $\deg (Y)=1$. In particular, $Y=\P^{r-c}$ is a linear subspace and hence $X$ is contained in the proper linear subspace $ \left\langle M,Y \right\rangle = \P^{r-t}$ of $\P^r$. Obviously, this is a contradiction since $X \subset \P^r$ is nondegenerate. In consequence, $t=0$ and hence $X \cap \Lambda$ spans $\Lambda$.
\end{proof}

Next, we will briefly recall the so-called \textit{split variety} which is also called a \textit{variety of completely decomposable forms}. A crucial connection between a split variety and the set $\Phi (\Gamma)_m$ is given in Lemma \ref{lem:deg splitting variety} below.

\begin{notation and remarks}\label{n and r:splitting variety}
Let $n$ and $\ell$ be positive integers and let $S=k[x_0 , x_1 , \ldots , x_n ]$.
\renewcommand{\descriptionlabel}[1]%
{\hspace{\labelsep}\textrm{#1}}
\begin{description}
    \setlength{\labelwidth}{13mm}
    \setlength{\labelsep}{1.5mm}
    \setlength{\itemindent}{0mm}

\item[$(1)$] Let $\P (S_{\ell} ) = \P^{{{n+\ell} \choose {n}} -1}$ be the projective space which parameterizes hypersurfaces of degree $\ell$ in $\P^n$. The subset
\begin{equation*}
\mbox{Split}_{\ell} (\P^n)= \{[F] ~|~ F \in \Phi_{\ell} \} \subset \P (S_{\ell} ).
\end{equation*}
Due to \cite{AB}, we call $\mbox{Split}_{\ell} (\P^n)$ \textit{a split variety}. In \cite{A}, this set is called \textit{the variety of completely decomposable $\ell$-forms}.

\item[$(2)$] The split variety $\mbox{Split}_{\ell} (\P^n )$ is the image of $\ell$ copies of $\P^n$ to $\P (S_{\ell})$ under the morphism
          \begin{equation*}
              \varphi : (\P^n )^{\ell} = (\P (S_1 ) )^{\ell} \rightarrow \P (S_{\ell} )
          \end{equation*}
          given by $\varphi ([H_1 ], \ldots , [H_{\ell} ] ) = [H_1 \cdots H_{\ell} ]$. Indeed, $\mbox{Split}_{\ell} (\P^n )$ is the image of $\varphi$ and
          \begin{equation*}
              \varphi : (\P^n )^{\ell}  \rightarrow \mbox{Split}_{\ell} (\P^n )
          \end{equation*}
          is a surjective finite morphism of degree $\ell !$. In consequence,
          \begin{equation*}
              \mbox{Split}_{\ell} (\P^n ) \subset \P (S_{\ell} )
          \end{equation*}
          is an $\ell n$-dimensional nondegenerate irreducible projective variety.

\item[$(3)$] Let $\sigma ((\P^n )^{\ell} ) \subset \P^{(n+1)^{\ell} -1}$ denote the Segre embedding of $(\P^n )^{\ell}$. Then $\varphi : (\P^n )^{\ell}  \rightarrow \mbox{Split}_{\ell} (\P^n )$ can be interpreted as the restriction of a linear projection
\begin{equation*}
\pi_{M} : \P^{(n+1)^{\ell} -1} - M \rightarrow \P (S_{\ell} )
\end{equation*}
to $\sigma ((\P^n )^{\ell} )$ where $M$ is subspace of $\P^{(n+1)^{\ell} -1}$ of dimension
\begin{equation*}
p (n,\ell) :=(n+1)^{\ell} - {{n+\ell} \choose {n}} -1
\end{equation*}
which does not meet with $\sigma ((\P^n )^{\ell} )$. Therefore it follows that
\begin{equation*}
\deg (\mbox{Split}_{\ell} (\P^n )) = \frac{\deg ((\P^n )^{\ell} )}{\ell !}  = \frac{(\ell n)!}{(n! )^{\ell} \times (\ell ! )} .
\end{equation*}
\end{description}
\end{notation and remarks}
\smallskip

\begin{lemma}\label{lem:deg splitting variety}
Let $\Gamma \subset \P^n$ be as Theorem \ref{thm:d=mn} and consider the subset
\begin{equation*}
[\Phi (\Gamma)_m ] := \{ [F] ~|~ F \in \Phi (\Gamma)_m \}
\end{equation*}
of $\P (S_m )$ corresponding to $\Phi (\Gamma)_m \subset S_m$. Then
\begin{equation*}
|[ \Phi (\Gamma)_m ] | = \deg ( \mbox{Split}_{m} (\P^n )).
\end{equation*}
\end{lemma}

\begin{proof}
Since $\Gamma$ is in linearly general position, each $F = H_1 \cdots H_m$ in $\Phi (\Gamma)_m$ defines the partition
\begin{equation*}
\Gamma = (\Gamma \cap V(H_1 ) ) \cup \cdots \cup (\Gamma \cap V(H_{m} ) )
\end{equation*}
of $\Gamma$ such that $|\Gamma \cap V(H_i )| =n$ for all $1 \leq i \leq m$. Conversely, for a partition
\begin{equation*}
\Gamma = \Gamma_1 \cup \cdots \cup \Gamma_m
\end{equation*}
of $\Gamma$ such that $|\Gamma_i| =n$ for all $1 \leq i \leq m$, let $H_i \in S_1$ be an equation of the hyperplane $\langle \Gamma_i \rangle$ in $\P^n$. Then $[H_1 \cdots H_m ]$ is an element of $[ \Phi (\Gamma)_m ]$. Therefore there is a one-to-one correspondence between $[ \Phi (\Gamma)_m ]$ and the set of partitions of $\Gamma$ as the union of its $m$ subsets having exactly $n$ elements. This implies that
\begin{equation*}
|[ \Phi (\Gamma)_m ] | = \frac{(m n)!}{(n! )^{m} \times (m! )}
\end{equation*}
and hence it is equal to $\deg ( \mbox{Split}_{m} (\P^n ))$ by Notation and Remarks \ref{n and r:splitting variety}(3).
\end{proof}

Now we are ready to give a proof of Theorem \ref{thm:d=mn}.\\

\noindent{\bf Proof of Theorem \ref{thm:d=mn}.}
Since $\Gamma \subset \P^n$ is in linearly general position and $d=mn$, we know that $\Gamma$ is $m$-normal. Thus, if we regard $\Pi := \P (I(\Gamma)_m )$ as a subspace of $\P (S_m)$, then
\begin{equation*}
\Pi \cap \mbox{Split}_{m} (\P^n ) = [\Phi (\Gamma)_m ].
\end{equation*}
and
\begin{equation*}
    \begin{CD}
        \dim~\Pi  &\quad = \quad& h^0 (\P^n , \mathcal{I}_{\Gamma} (m))-1  \\
        &\quad = \quad& {{m+n} \choose {n}} -mn-1  \\
        &\quad = \quad& \mbox{codim} (\mbox{Split}_{m} (\P^n ) , \P (S_m )) .
    \end{CD}
\end{equation*}
Also $\Phi (\Gamma)_m$ generates $I(\Gamma)_m$ as a $k$-vector space if and only if the set $[\Phi (\Gamma)_m ]$ spans $\Pi$. To show the second statement, we will use Notation and Remarks \ref{n and r:splitting variety}. Let $\Lambda := \langle \Pi , M \rangle$. Then
\begin{equation*}
    \begin{CD}
        \dim ~ \Lambda & \quad = \quad & \dim ~\Pi + \dim~M + 1  \\
        & \quad = \quad & {{m+n} \choose {n}} -mn-1  + p(n,m)+1 \\
        & \quad = \quad & (n+1)^m - mn -1 \\
        &\quad = \quad& \mbox{codim} ( \sigma ( (\P^n )^m ) , \P^{(n+1)^m -1})
    \end{CD}
\end{equation*}
and
\begin{equation*}
    \Lambda \cap \sigma ((\P^n )^{m} ) = \varphi^{-1} ([\Phi (\Gamma)_m ]).
\end{equation*}
Also, $\varphi^{-1} ([\Phi (\Gamma)_m ])$ contains exactly $m ! \times |[\Phi (\Gamma)_m ]|$ elements since every element in $\Phi (\Gamma)_m$ is the product of $m$ distinct linear forms. By Lemma \ref{lem:deg splitting variety}, it holds that
\begin{equation*}
    \begin{CD}
        m ! \times |[\Phi (\Gamma)_m ]| & \quad = \quad & m ! \times \deg ( \mbox{Split}_{m} (\P^n ))\\
        & \quad = \quad & \deg ( \sigma ((\P^n )^{m} )  )
    \end{CD}
\end{equation*}
In consequence, $\dim ~ \Lambda$ and $|\sigma ( (\P^n )^m \cap \Lambda|$ are respectively equal to the codimension and the degree of $\sigma ( (\P^n )^m)$ in $\P^{(n+1)^m -1}$. Since $\sigma ( (\P^n )^m)$ is locally Cohen-Macaulay, it follows by Proposition \ref{prop:lCM flatness} that
\begin{equation*}
\sigma ( (\P^n )^m) \cap \Lambda  = \varphi^{-1} ([\Phi (\Gamma)_m ])
\end{equation*}
spans $\Lambda$. Obviously this can happen only when $[\Phi (\Gamma)_m ]$ spans $\Pi$. This completes the proof that $\Phi (\Gamma)_m$ generates $I(\Gamma)_m$ as a $k$-vector space.    \qed \\

\section{Completely decomposable equations II : Arbitrary case}
\noindent As was mentioned in the beginning of $\S~2$, our aim of this section is to verify that $\Phi (\Gamma)_{\ell}$ generates $I(\Gamma)_{\ell}$ as a $k$-vector space if $\ell n \geq d$. More precisely, our purpose in this section is to prove the following theorem.

\begin{theorem}\label{thm:generation}
Let $\Gamma \subset \P^n$ be a finite set of $d$ points in linearly general position. Then
\begin{equation*}
\Phi (\Gamma)_{\ell} \quad \begin{cases}
\mbox{is empty}  & \mbox{if} \quad \ell < \lceil \frac{d}{n} \rceil, \mbox{and} \\
\mbox{generates $I(\Gamma)_{\ell}$ as a $k$-vector space} & \mbox{if} \quad \ell \geq \lceil \frac{d}{n} \rceil. \end{cases}
\end{equation*}
(Here, $\lceil x \rceil$ denotes the smallest integer $\geq x$.)
\end{theorem}

The following elementary lemma allows one to deduce the proof of Theorem \ref{thm:generation} from the special case where $d=mn$ and $\ell = m$ which was dealt with in Theorem \ref{thm:d=mn}.

\begin{lemma}\label{lem:induction}
Let $\Gamma_0  \subset \P^n$ be a finite set of $d+2$ points in linearly general position. For two distinct points $x$ and $y$ in $\Gamma_0$, let
\begin{equation*}
\Gamma_1 := \Gamma_0 - \{x  \} , \quad \Gamma_2 := \Gamma_0 - \{y  \} \quad \mbox{and} \quad \Gamma  := \Gamma_0 - \{x ,y \}.
\end{equation*}
If $\ell \geq  \lceil \frac{d+1}{n} \rceil$, then $I(\Gamma)_\ell = I(\Gamma_1)_\ell + I(\Gamma_2)_\ell$.
\end{lemma}

\begin{proof}
First note that $d+1 \leq \ell n$ and hence $|\Gamma_0 | = d+2 \leq \ell n+1$. So, all of $\Gamma_0$, $\Gamma_1$, $\Gamma_2$ and $\Gamma$ satisfy the $\ell$-normality. In particular, it holds that
\begin{equation*}
\dim_k ~ I(\Gamma_1)_\ell = \dim_k ~ I(\Gamma_2)_\ell = \binom{n+\ell}{n}-(d+1),
\end{equation*}
\begin{equation*}
\dim_k ~ I(\Gamma )_\ell = \binom{n+\ell}{n}-d
\end{equation*}
and
\begin{equation*}
\dim_k ~ I(\Gamma_0 )_\ell = \binom{n+\ell}{n}-(d+2).
\end{equation*}
Also $I(\Gamma_1)_\ell \cap I(\Gamma_2)_\ell \subseteq I(\Gamma_0 )_\ell$ because $\Gamma_1 \cup \Gamma_2 = \Gamma_0$. Therefore we get
\begin{equation*}
\begin{CD}
\dim_k ~ I(\Gamma_1)_\ell + I(\Gamma_2)_\ell &\quad = \quad & \dim_k ~ I(\Gamma_1)_\ell + \dim_k ~ I(\Gamma_2)_\ell - \dim_k ~ (I(\Gamma_1)_\ell \cap I(\Gamma_2)_\ell)\\
&\quad=\quad& \binom{n+\ell}{n}-d \\
 &\quad=\quad& \dim_k ~ I(\Gamma)_\ell .
\end{CD}
\end{equation*}
This implies the desired equality
\begin{equation*}
I(\Gamma)_\ell=I(\Gamma_1)_\ell+I(\Gamma_2)_\ell
\end{equation*}
since $I(\Gamma_1)_\ell + I(\Gamma_2)_\ell$ is a subspace of $ I(\Gamma)_\ell$.
\end{proof}

To simplify the proof of Theorem \ref{thm:generation}, we introduce a definition.

\begin{definition}
We will say that \textit{property $P(n,d,\ell)$} holds if for any finite set $\Gamma \subset \P^n$ of $d$ points in linearly general position, $I(\Gamma)_{\ell}$ is generated by $\Phi (\Gamma)_{\ell}$ as a $k$-vector space.
\end{definition}

For example, Theorem \ref{thm:d=mn} says that property $P(n,mn,m)$ always holds.

Now, we give a proof of Theorem \ref{thm:generation}.\\

\noindent {\bf Proof of Theorem \ref{thm:generation}.} Put $m = \lceil \frac{d}{n} \rceil$. Thus $(m-1)n+1 \leq d \leq mn$.

First, suppose that $\Phi (\Gamma)_{\ell}$ is nonempty and let $H_1 \cdots H_{\ell}$ be an element of $\Phi (\Gamma)_{\ell}$. Thus
\begin{equation*}
\Gamma = (\Gamma \cap V(H_1 ) ) \cup \cdots \cup (\Gamma \cap V(H_{\ell} ) ).
\end{equation*}
Also
\begin{equation*}
|\Gamma \cap V(H_i )| \leq n \quad \mbox{for all} \quad 1 \leq i \leq \ell
\end{equation*}
since $\Gamma$ is in linearly general position. In particular, it holds that
\begin{equation*}
d=|\Gamma| \leq \sum_{i=1} ^{\ell} |\Gamma \cap V(H_i ) | \leq \ell n,
\end{equation*}
or equivalently, $\ell \geq m$. This shows that $\Phi (\Gamma)_{\ell}$ is empty if $\ell < m$ and hence it proves the first part of our theorem.

Now, suppose that $\ell \geq m$. Write $d = \ell n - \epsilon$ for some $\epsilon \geq 0$. To show that property $P(n,d,\ell)$ holds, we will use the induction on $\epsilon \geq 0$.

If $\epsilon =0$ and so $d=\ell n$, then property $P(n,d,\ell)$ holds by Theorem \ref{thm:d=mn}.

Next, we assume that $\epsilon$ is positive. Choose two distinct points $x,y \in \P^n - \Gamma$ such that the finite set
\begin{equation*}
\Gamma_0 := \Gamma \cup \{x ,y \} ~\subset ~\P^n
\end{equation*}
is in linearly general position. For
\begin{equation*}
\Gamma_1 := \Gamma \cup \{x \} \quad \mbox{and} \quad \Gamma_2 := \Gamma \cup \{y \},
\end{equation*}
we have
\begin{equation*}
|\Gamma_1 | = |\Gamma_2 | = \ell n - (\epsilon -1).
\end{equation*}
Thus, by induction hypothesis, it follows that $I(\Gamma_1 )_m$ and $I(\Gamma_2 )_m$ are respectively generated by $\Phi (\Gamma_1 )_m$ and $\Phi (\Gamma_2 )_m$ as $k$-vector spaces. Since
\begin{equation*}
I(\Gamma)_{\ell} = I(\Gamma_1 )_{\ell} + I(\Gamma_2 )_{\ell}
\end{equation*}
by Lemma \ref{lem:induction} and since $\Phi (\Gamma )_{\ell}$ contains the union $\Phi (\Gamma_1 )_{\ell} \cup \Phi (\Gamma_2 )_{\ell}$, we can conclude that $I(\Gamma )_{\ell}$ is generated by $\Phi (\Gamma )_{\ell}$ as a $k$-vector space. This completes the proof that property $P(n,d,\ell)$ holds. \qed \\

\section{Completely decomposable generators of $I(\Gamma)$ when $d=mn$}
\noindent Let $\Gamma \subset \P^n$ be a finite set of $d$ points in linearly general position. Through this section and the next section, we will prove the equality
\begin{equation*}
I(\Gamma) = \langle I(\Gamma)_{\leq m-1} , \Phi (\Gamma)_m \rangle
\end{equation*}
mentioned in Theorem \ref{thm:main}.

In this section we concentrate on verifying the following theorem about the case where $d=mn$.

\begin{theorem}\label{thm:d=mn generation of the ideal}
Let $\Gamma \subset \P^n$ be as above such that $d=mn$ for some $m \geq 2$. Then $I(\Gamma)_{m+1}$ is generated by the set
\begin{equation*}
S_1 \Phi (\Gamma)_m := \{ hF~|~ h\in S_1 - \{0\} ,~F \in \Phi (\Gamma)_m \}
\end{equation*}
as a $k$-vector space.
\end{theorem}

To this aim, we need a few notation and definition.

\begin{definition and remark}\label{d and r:intersection with Gamma}
Let $\Gamma \subset \P^n$ be as a finite set of $d$ points.
\smallskip

\renewcommand{\descriptionlabel}[1]%
{\hspace{\labelsep}\textrm{#1}}
\begin{description}
    \setlength{\labelwidth}{13mm}
    \setlength{\labelsep}{1.5mm}
    \setlength{\itemindent}{0mm}

\item[$(1)$] For a homogeneous polynomial $F \in S$, we define $V_\Gamma(F)$ and $\lambda_{\Gamma} (F)$ respectively as
\begin{equation*}
V_\Gamma (F) := \Gamma \cap V(F) \quad \mbox{and} \quad \lambda_\Gamma (F) = |V_\Gamma (F)|.
\end{equation*}
Obviously, $\lambda_\Gamma (F) \leq d$ and equality is attained if and only if $F \in I(\Gamma)$.

\item[$(2)$] Let $F_1,\cdots, F_{\ell}$ be nonzero homogeneous polynomials in $S$. Then it holds that
\begin{equation*}
V_\Gamma (F_1\cdots F_{\ell} ) = \bigcup_{i=1}^{\ell} V_\Gamma (F_i)
\end{equation*}
and hence
\begin{equation}\label{eq:4.1}
\lambda_\Gamma (F_1 \cdots F_{\ell})  \leq \sum_{i=1}^{\ell} |V_\Gamma(F_i)| = \sum_{i=1}^{\ell} \lambda_{\Gamma} (F_i).
\end{equation}

\item[$(3)$] Let $F$ be an element of $\Phi (\Gamma)_{\ell+1}$. Thus $F= h_1 \cdots h_{\ell+1}$ for some nonzero linear forms $h_1 ,\ldots , h_{\ell}, h_{\ell+1} \in S_1$. Then we define $\mu_{\Gamma} (F)$ and $\nu_{\Gamma , F} (h_i)$, $1 \leq i \leq \ell+1$, respectively as
\begin{equation*}
\mu_{\Gamma} (F) = \mbox{max} \left \{ \lambda_{\Gamma} \left( \frac{F}{h_i} \right) ~|~ 1 \leq i \leq \ell+1 \right \}
\end{equation*}
and
\begin{equation*}
\nu_{\Gamma ,F} (h_i) =  \Gamma - V_{\Gamma}  \left( \frac{F}{h_i} \right) \subseteq V_{\Gamma} (h_i).
\end{equation*}
Note that $\mu_{\Gamma} (F) \leq d$ and equality is attained if and only if $\frac{F}{h_i} \in \Phi (\Gamma)_{\ell}$ for some $i$. Also $\nu_{\Gamma , F} (h_1),\cdots, \nu_{\Gamma , F} (h_{\ell+1})$ are pairwise disjoint and so it holds that
\begin{equation}\label{eq:basic inequality}
\sum_{i=1}^{\ell+1} |\nu_{\Gamma , F} (h_i)|=|\bigcup_{i=1}^{\ell+1} \nu_{\Gamma , F} (h_i)| \leq |\bigcup_{i=1}^{\ell+1} V_{\Gamma} (h_i)|  = | V_{\Gamma} (F)| = |\Gamma| =d .
\end{equation}
\end{description}
\end{definition and remark}
\smallskip

\begin{lemma}\label{lem:main_1}
Let $\Gamma \subset \P^n$ be a finite set of $2n$ points in linearly
general position and let $F \in \Phi (\Gamma)_3$ be such that
$\mu_{\Gamma} (F) = 2n-1$. Then $F$ is contained in the subspace of
$I(\Gamma)_3$ generated by the subset
\begin{equation*}
S_1 \Phi (\Gamma)_2 := \{ hF~|~ h\in S_1 - \{0\} ,~F \in \Phi
(\Gamma)_2 \}
\end{equation*}
as a $k$-vector space.
\end{lemma}

\begin{proof}
Let $\Gamma = \{ P_1 , \ldots , P_n , Q_1 , \ldots , Q_n \}$ and
write $F=h_1 h_2 h_3$ where $h_1$, $h_2$ and $h_3$ are nonzero linear forms.
Without loss of generality, we may assume the following conditions:\\

\begin{enumerate}
\item[$(i)$] $\Gamma - V_{\Gamma} (h_1 h_2 ) = \{ Q_n \}$ and hence $\lambda_{\Gamma} (h_1 h_2) = 2n-1$
\item[$(ii)$] $V_{\Gamma} (h_1 )=\{ P_1 , \ldots , P_n \}$ and $V_{\Gamma} (h_2 )$ is equal to
either $\{ P_n , Q_1 , \ldots , Q_{n-1} \}$ or else $\{ Q_1 , \ldots , Q_{n-1} \}$ \\
\end{enumerate}

\noindent \textsl{Case 1.} Suppose that $V_{\Gamma} (h_2 ) = \{ P_n
, Q_1 , \ldots , Q_{n-1} \}$. We may assume that $P_n , Q_1 , \ldots
, Q_{n-1}$ are the first $n$ coordinate points of $\P^n$ in their order.
Choose $h \in S_1$ such that $V_{\Gamma} (h) = \{P_1 , \ldots ,
P_{n-1} , Q_n\}$. Note that
\begin{equation*}
hh_2 \in \Phi (\Gamma)_2 \quad \mbox{and} \quad h(P) \neq 0 \quad \mbox{for every} \quad P \in V_{\Gamma} (h_2 ).
\end{equation*}
Now, consider the linear form $h'  = b_1 x_1 + \cdots + b_{n-1} x_{n-1}$ where
\begin{equation*}
b_i := \frac{h_1(Q_i)h_3(Q_i)}{h (Q_i)} \quad \mbox{for} \quad
i=1,\ldots ,n-1.
\end{equation*}
Then one can easily check that $G := h_1 h_3 - h h'$ is an element
of $I(\Gamma)_2$. Thus $G$ is a $k$-linear combination of elements
of $\Phi (\Gamma)_2$ by Theorem \ref{thm:d=mn}. Then
\begin{equation*}
F = h_1 h_2 h_3 = h_2 G + h h_2 h'
\end{equation*}
is contained in the subspace generated by $S_1 \Phi (\Gamma)_2$
since $hh_2 \in \Phi (\Gamma)_2$. \\

\noindent \textsl{Case 2.} Suppose that $V_{\Gamma} (h_2 ) = \{Q_1 ,
\ldots , Q_{n-1} \}$. Letting
\begin{equation*}
A=\{Q_1,\cdots,Q_{n-1},P_1\} \quad \mbox{and} \quad
B=\{Q_1,\cdots,Q_{n-1},P_2\},
\end{equation*}
it holds by Lemma \ref{lem:induction} that
\begin{equation*}
I(\{Q_1,\cdots,Q_{n-1}\})_1 = I(A)_1 + I(B)_1 .
\end{equation*}
In particular, we can write $h_2$ as
\begin{equation*}
h_2 = h_{2,1} + h_{2,2} \quad \mbox{for some} \quad h_{2,1} \in
I(A)_1 \quad \mbox{and} \quad h_{2,2} \in I(B)_1.
\end{equation*}
Since both of $h_1 h_{2,1} h_3$ and $h_1 h_{2,2} h_3$ correspond to \textsl{Case 1}, we can conclude that
\begin{equation*}
h_1 h_2 h_3 = h_1 h_{2,1} h_3 + h_1 h_{2,2} h_3
\end{equation*}
is contained in the subspace generated by $S_1 \Phi (\Gamma)_2$.
\end{proof}

\begin{lemma}\label{lem:mu=d-1 case}
Let $\Gamma \subset \P^n$ be a finite set of $mn$ points in linearly general position for some $m \geq 3$ and let $F \in \Phi(\Gamma)_{m+1}$ be such that $\mu(F)=mn-1$. Then $F$ is contained in the $k$-vector space generated by $S_1 \Phi (\Gamma)_{m}$.
\end{lemma}

\begin{proof}
Write $F= h_1 \cdots h_{m+1}$ where $h_1 ,\ldots , h_m, h_{m+1}$ are nonzero linear forms. Since $\mu(F)=d-1=mn-1$, we may assume that
\begin{equation*}
\mu_{\Gamma} (F)=\lambda_{\Gamma} (h_1\cdots h_m) = mn-1.
\end{equation*}
Thus the set $\nu_{\Gamma , F} (h_{m+1} ) = \Gamma - V_{\Gamma} (h_1\cdots h_m)$ consists of a single point, say $Q$. Note that $\lambda_{\Gamma} (h_i) \leq n$ for all $1 \leq i \leq m+1$ since $\Gamma$ is in linearly general position. By using (\ref{eq:4.1}) we have
\begin{equation*}
mn-1 = \lambda_{\Gamma} (h_1\cdots h_m)  \leq \sum_{i=1} ^m \lambda_{\Gamma} (h_i)  \leq mn.
\end{equation*}
If $\lambda_{\Gamma} (h_i)\leq n-1$ for more than two $i$'s, then
\begin{equation*}
\lambda_{\Gamma} (h_1\cdots h_m) \leq \sum_{i=1}^{m}\lambda_{\Gamma} (h_i)\leq mn-2.
\end{equation*}
Therefore, at least $m-1$ of $\lambda(h_1),\cdots,\lambda(h_m)$ must be equal to $n$. Without loss of generality, we may assume the followings:\\

\begin{enumerate}
\item[$(i)$] $\lambda_{\Gamma} (h_1 ) = \cdots = \lambda_{\Gamma} (h_{m-1} ) =n$ and $n-1 \leq \lambda_{\Gamma} (h_m ) \leq n$;
\item[$(ii)$] Either
   \begin{enumerate}
   \item[$(ii.1)$] $\lambda_{\Gamma} (h_m ) = n-1$ and $V_{\Gamma} (h_1 ) , \ldots , V_{\Gamma} (h_m )$ are pairwise disjoint
   \end{enumerate}
or else
   \begin{enumerate}
   \item[$(ii.2)$] $\lambda(h_m ) = n$, the intersection $V_{\Gamma} (h_{m-1} ) \cap V_{\Gamma} (h_m )$ is exactly a point, say $P$, and the sets
       \begin{equation*}
      V_{\Gamma} (h_1 ) , \ldots , V_{\Gamma} (h_{m-1} ) ,V_{\Gamma} (h_m ) - \{ P \}
       \end{equation*}
       are pairwise disjoint.\\
   \end{enumerate}
\end{enumerate}
Then, we have $\lambda_{\Gamma} (h_{m-1} h_m)=2n-1$. Note that the set
\begin{equation*}
\Gamma_0 := V_{\Gamma} (h_{m-1} h_m) \cup \{Q \} \subset \P^n
\end{equation*}
is a finite set of $2n$ points in linearly general position and hence
\begin{equation*}
G:=h_{m-1} h_m h_{m+1} \in \Phi (\Gamma_0 )_3 \quad \mbox{and} \quad \mu_{\Gamma_0} (G ) = \lambda_{\Gamma_0} (h_{m-1} h_m ) = 2n-1.
\end{equation*}
Thus, it holds by Lemma \ref{lem:main_1} that $G$ is contained in the $k$-vector space generated by $S_1 \Phi (\Gamma_0 )_2$. This shows that $F = h_1 \cdots h_{m-2} \times G$ is contained in the $k$-vector space generated by
\begin{equation*}
h_1 \cdots h_{m-2} \times S_1 \Phi (\Gamma_0 )_2 := \{ h_1 \cdots h_{m-2} \times g ~|~ g \in S_1 \Phi (\Gamma_0 )_2 \}.
\end{equation*}
Also the set
\begin{equation*}
h_1 \cdots h_{m-2} \times \Phi (\Gamma_0 )_2 := \{ h_1 \cdots h_{m-2} \times q ~|~ q \in \Phi (\Gamma_0 )_2 \}
\end{equation*}
is a subset of $\Phi (\Gamma)_m$ since $\Gamma$ is equal to the union $\Gamma_0 \cup V_{\Gamma} (h_1 \cdots h_{m-2})$. In consequence, it is shown that $F$ is contained in the $k$-vector space generated by $S_1 \Phi (\Gamma)_m$. This completes the proof.
\end{proof}

Now, we are ready to give a proof of Theorem \ref{thm:d=mn generation of the ideal}.\\

\noindent {\bf Proof of Theorem \ref{thm:d=mn generation of the ideal}.}
Throughout the proof, we will denote by $V$ the $k$-vector space generated by $S_1 \Phi (\Gamma)_m$.

By Theorem \ref{thm:generation}, the set $\Phi (\Gamma)_{m+1}$ generates $I(\Gamma)_{m+1}$ as a $k$-vector space. Thus it is enough to show that every element $F$ in $\Phi (\Gamma)_{m+1}$ is contained in $V$. To this aim, we use induction on the value
\begin{equation*}
\epsilon_{\Gamma} (F) := mn - \mu_{\Gamma} (F) \geq 0.
\end{equation*}
Let $F$ be an element of $\Phi (\Gamma)_{m+1}$. Thus $F= h_1 \cdots h_{m+1}$ for some nonzero linear forms $h_1 ,\ldots , h_{m}, h_{m+1} \in S_1$. Note that if $\epsilon_{\Gamma} (F) = 0$ then $F$ is already an element of $S_1 \Phi (\Gamma)_m$. Also if $\epsilon_{\Gamma} (F) = 1$ and hence $\mu_{\Gamma} (F) = mn-1$, then $F$ is contained in $V$ by Lemma \ref{lem:mu=d-1 case}.

Now, suppose that $\epsilon_{\Gamma} (F) \geq 2$. Without loss of generality, we may assume that
\begin{equation*}
\mu_{\Gamma} (F) = \lambda_{\Gamma} (h_1 \cdots h_m ).
\end{equation*}
Let $A$ denote the finite set $V_{\Gamma} ( h_1 \cdots h_m)$. Thus $\mu_{\Gamma} (F) = |A| \leq mn-2$. Then there exists $i \in \{1, \ldots , m \}$ such that
\begin{equation*}
|\nu_{\Gamma,F} (h_i )| \leq n-1
\end{equation*}
(cf. see (\ref{eq:basic inequality})). We may assume that $i=m$. Let
\begin{equation*}
\nu_{\Gamma,F} (h_m ) = \{ Q_1 , \ldots , Q_{\delta} \}
\end{equation*}
where $\delta = |\nu_{\Gamma,F} (h_i )| \leq n-1$. Now, choose two distinct points $P_1$ and $P_2$ in $\Gamma -A$. Then
\begin{equation*}
h_{m+1} (P_1 ) = h_{m+1} (P_2 )=0
\end{equation*}
since $F \in \Phi(\Gamma)_{m+1}$. Also, it holds by Lemma \ref{lem:induction} that
\begin{equation*}
I(\{ Q_1 , \ldots , Q_{\delta} \} )_1 = I(\{ Q_1 , \ldots , Q_{\delta} , P_1 \} )_1 + I(\{ Q_1 , \ldots , Q_{\delta} , P_2 \} )_1 .
\end{equation*}
In particular, we can write $h_m = h_{m,1} + h_{m,2}$ for some $h_{m,j} \in I(\{ Q_1 , \ldots , Q_{\delta} , P_j \} )$. Then
\begin{equation*}
F = G_1 + G_2
\end{equation*}
where $G_j := h_1 \cdots h_{m-1} h_{m,j} h_{m+1}$ for each $j=1,2$. Also one can easily check that $G_1$ and $G_2$ are contained in $\Phi (\Gamma)_{m+1}$. Furthermore, it holds that
\begin{equation*}
\mu_{\Gamma} (G_j ) \geq \lambda_{\Gamma} (h_1 \cdots h_{m-1} h_{m,j}) \geq \lambda_{\Gamma} (h_1 \cdots h_{m-1} h_m) +1 = \mu_{\Gamma} (F) +1
\end{equation*}
since $h_{m,j} (P_j )=0$ while $h_m (P_j ) \neq 0$. Now, by induction hypothesis it follows that $G_1$ and $G_2$ are contained in $V$. This implies that $F = G_1 +G_2$ is an element of $V$ and hence $V$ is equal to $I(\Gamma)_{m+1}$. \qed \\

\section{Proof of Main Theorem}
\noindent This section is devoted to giving a proof of Theorem \ref{thm:main}. We begin with the following lemma which says that in order to prove Theorem \ref{thm:main} it suffices to consider only the degree $m$ and $m+1$ pieces of $I(\Gamma)$.

\begin{lemma}\label{lem:equivalet condition}
Let $\Gamma \subset \P^n$ be a finite set of $d$ points in linearly general position. Suppose that $d \leq mn$. Then the following two statements are equivalent:
\begin{enumerate}
\item[$(i)$] $I(\Gamma)$ is generated by forms of degree $\leq m$.
\item[$(ii)$] $\Phi (\Gamma)_{m+1}$ is contained in the $k$-vector space generated by the set
\begin{equation*}
S_1 \Phi (\Gamma)_{m} := \{ h F ~|~ h \in S_1 \quad \mbox{and} \quad F \in \Phi (\Gamma)_{m} \}.
\end{equation*}
\end{enumerate}
\end{lemma}

\begin{proof}
It holds by Theorem \ref{thm:generation} that $I(\Gamma)_m$ and $I(\Gamma)_{m+1}$ are respectively generated by $\Phi (\Gamma)_m$ and $\Phi (\Gamma)_{m+1}$ as $k$-vector spaces.

$((i) \Rightarrow (ii)$ : If $I(\Gamma)$ is generated by forms of degree $\leq m$ then the $k$-vector space generated by $S_1 I(\Gamma)_{m}$ is equal to $I(\Gamma)_{m+1}$. Therefore $(i)$ implies $(ii)$.

$(ii) \Rightarrow (i)$ : Since $d \leq mn+1$, we know that the Castelnuovo-Mumford regularity of $\Gamma$ is at most $m+1$ and hence $I(\Gamma)$ is generated by forms of degree $\leq m+1$. So, it needs to show that $I(\Gamma)_{m+1}$ is spanned by the set
\begin{equation*}
S_1 I(\Gamma)_m := \{ h F ~|~ h \in S_1 \quad \mbox{and} \quad F \in I(\Gamma)_{m} \}.
\end{equation*}
By our assumption $(ii)$, it follows that $S_1 \Phi (\Gamma)_{m}$ generates $I(\Gamma)_{m+1}$ as a $k$-vector space. This completes the proof that $(ii)$ implies $(i)$.
\end{proof}

Now, we give a proof of our main theorem in the present paper. \\

\noindent {\bf Proof of Theorem \ref{thm:main}.} Let $V$ denote the $k$-vector space generated by $S_1 \Phi (\Gamma)_m$.

First we will show that $\Phi (\Gamma)_{m+1}$ is a subset of $V$ .To this aim, let us write $d = mn-\epsilon$ where $\epsilon$ is a nonnegative integer. We use induction on $\epsilon$. If $\epsilon = 0$ and so $d=mn$, then we are done by Theorem \ref{thm:d=mn generation of the ideal}. Now, suppose that $\epsilon >0$ and choose two distinct points $x,y \in \P^n - \Gamma$ such that $\Gamma_0 := \Gamma \cup \{x ,y \}$ is in linearly general position. Let
\begin{equation*}
\Gamma_1 := \Gamma \cup \{x \} \quad \mbox{and} \quad \Gamma_2 := \Gamma \cup \{y \}.
\end{equation*}
Then
\begin{equation*}
|\Gamma_1 | = |\Gamma_2 | = d+1 \leq mn
\end{equation*}
and hence it holds by induction hypothesis that $I(\Gamma_1 )_{m+1}$ and $I(\Gamma_2 )_{m+1}$ are respectively generated by $S_1 \Phi (\Gamma_1 )_m$ and $S_1 \Phi (\Gamma_2 )_m$. Also
\begin{equation*}
I(\Gamma)_{m+1} = I(\Gamma_1 )_{m+1} + I(\Gamma_2 )_{m+1}
\end{equation*}
by Lemma \ref{lem:induction}. Therefore $I(\Gamma)_{m+1}$ is generated by the union of $S_1 \Phi (\Gamma_1 )_m$ and $S_1 \Phi (\Gamma_2 )_m$. Obviously,
$S_1 \Phi (\Gamma )_m$ contains the union
\begin{equation*}
S_1 \Phi (\Gamma_1 )_m \cup S_1 \Phi (\Gamma_2 )_m
\end{equation*}
In consequence, it is shown that $I(\Gamma)_{m+1}$ is generated by $S_1 \Phi (\Gamma )_m$. In particular, $\Phi (\Gamma )_{m+1}$ is contained in $V$, the $k$-vector space spanned by $S_1  \Phi (\Gamma )_m$.

Now, by Lemma \ref{lem:equivalet condition}, it follows that $I(\Gamma)$ is generated by forms of degree $\leq m$ and hence
\begin{equation*}
I(\Gamma) = \langle I(\Gamma)_{\leq m-1} ,I(\Gamma)_m \rangle .
\end{equation*}
Since we know already that $I(\Gamma)_m$ is generated by $\Phi (\Gamma)_m$ as a $k$-vector space, it holds that $I(\Gamma) = \langle I(\Gamma)_{\leq m-1} ,\Phi (\Gamma)_m \rangle$. This finishes the proof.
\qed \\

\begin{remark}
Let $\Gamma \subset \P^n$ be a finite set of $d$ points in linearly general position such that
\begin{equation*}
(m-1)n \leq d \leq mn \quad \mbox{for some} \quad m \geq 2.
\end{equation*}
So, Theorem \ref{thm:main} says that
\begin{equation*}
I(\Gamma) = \langle I(\Gamma)_{\leq m-1} ,\Phi (\Gamma)_m \rangle .
\end{equation*}
Here we want to explain how to find an explicit finite generating set of $\Phi (\Gamma)_m$.
 \renewcommand{\descriptionlabel}[1]%
{\hspace{\labelsep}\textrm{#1}}
\begin{description}
    \setlength{\labelwidth}{13mm}
    \setlength{\labelsep}{1.5mm}
    \setlength{\itemindent}{0mm}

\item[$(1)$] Suppose that $d = mn$. Then there is a one-to-one correspondence between $[\Phi (\Gamma)_m ]$ and the set of partitions of $\Gamma$ as the union of $m$ subsets having exactly $n$ elements. For details, see the proof of Lemma \ref{lem:deg splitting variety}. So, we get a subset $\Sigma (\Gamma)$ of $\Phi (\Gamma)_m$ consisting of $\frac{(m n)!}{(n! )^{m} \times (m ! )}$ completely decomposable $m$-forms (cf. Notation and Remarks \ref{n and r:splitting variety}(3). Then it is shown in the proof of Theorem \ref{thm:d=mn} that $\Sigma (\Gamma)$ generates $I(\Gamma)_m$.

\item[$(2)$] Suppose that $d = mn-\epsilon$ for some $\epsilon \in \{1,\ldots , n-1 \}$. Then choose distinct $(\epsilon+1)$-points $x_1 , \ldots , x_{\epsilon+1}$ in $\P^n - \Gamma$ such that
    \begin{equation*}
    \Gamma_0 :=  \Gamma \cup \{ x_1 , \ldots , x_{\epsilon+1} \}
    \end{equation*}
is in linearly general position. Now, let
\begin{equation*}
\Gamma_j := \Gamma_0 - \{x_j \} \quad \mbox{for} \quad 1 \leq j \leq \epsilon +1.
\end{equation*}
Then $|\Gamma_j| =mn$ for every $ 1 \leq j \leq \epsilon +1$. Also one can show that
\begin{equation*}
\Gamma = \Gamma_1 \cap \cdots \cap \Gamma_{\epsilon} \cap \Gamma_{\epsilon+1}
\end{equation*}
and hence
\begin{equation*}
I(\Gamma)_m = I(\Gamma_1 )_m + \cdots + I(\Gamma_{\epsilon} )_m + I(\Gamma_{\epsilon+1} )_m
\end{equation*}
(cf. Lemma \ref{lem:induction}). Also, by $(1)$, $I(\Gamma_j )_m$ is generated by $\Sigma (\Gamma_j )$ consisting of $\frac{(m n)!}{(n! )^{m} \times (m ! )}$ completely decomposable $m$-forms for each $1 \leq j \leq \epsilon +1$. In consequence, we have the set
\begin{equation*}
\Sigma (\Gamma_1 ) \cup \cdots \cup \Sigma (\Gamma_{\epsilon} ) \cup \Sigma (\Gamma_{\epsilon+1} )
\end{equation*}
consisting of at most $(\epsilon+1) \times \frac{(m n)!}{(n! )^{m} \times (m ! )}$ completely decomposable $m$-forms which generates $I(\Gamma)_m$ as a $k$-vector space.
\end{description}
\end{remark}

\end{document}